\documentclass{amsart}
\usepackage[all]{xy}
\usepackage{verbatim}
\usepackage{color}
\usepackage{amsthm}
\usepackage{amssymb}
\usepackage[colorlinks=true]{hyperref}
\usepackage{footmisc}
\usepackage{stmaryrd}

\numberwithin{equation}{section}

\newtheorem{theorem}[equation]{Theorem}
\newtheorem*{theorem*}{Theorem} \newtheorem{lemma}[equation]{Lemma}

\newtheorem*{conjecture*}{Mamma Conjecture}
\newtheorem*{conjecture1*}{Mamma Conjecture (revisited)}
\newtheorem{proposition}[equation]{Proposition}
\newtheorem{corollary}[equation]{Corollary}
\newtheorem*{corollary*}{Corollary}

\theoremstyle{remark}

\newtheorem{example}[equation]{Example}

\newtheorem{notation}[equation]{Notation}

\theoremstyle{remark}
\newtheorem{remark}[equation]{Remark}

\newcommand{\cA}{{\mathcal A}}
\newcommand{\cB}{{\mathcal B}}
\newcommand{\cC}{{\mathcal C}}
\newcommand{\cD}{{\mathcal D}}

\newcommand{\cN}{{\mathcal N}}

\newcommand{\cT}{{\mathcal T}}

\newcommand{\cZ}{{\mathcal Z}}

\newcommand{\bbC}{\mathbb{C}}

\newcommand{\bbF}{\mathbb{F}}

\newcommand{\bbQ}{\mathbb{Q}}
\newcommand{\bbZ}{\mathbb{Z}}

\DeclareMathOperator{\id}{id}

\DeclareMathOperator{\NChow}{NChow} 
\DeclareMathOperator{\NNum}{NNum} 

\DeclareMathOperator{\Num}{Num} 

\newcommand{\dgcat}{\mathrm{dgcat}} 

\newcommand{\perf}{\mathrm{perf}}

\newcommand{\Chow}{\mathrm{Chow}}

\newcommand{\dg}{\mathrm{dg}}

\newcommand{\Hom}{\mathrm{Hom}}
\newcommand{\End}{\mathrm{End}}

\newcommand{\rep}{\mathrm{rep}}

\newcommand{\dgHo}{\mathrm{H}^0}

\newcommand{\Hmo}{\mathrm{Hmo}}
\newcommand{\op}{\mathrm{op}}

\newcommand{\too}{\longrightarrow}

\newcommand{\ie}{\textsl{i.e.}\ }
\newcommand{\eg}{\textsl{e.g.}}

\let\oldmarginpar\marginpar
\def\marginpar#1{\oldmarginpar{\tiny #1}}

\begin{document}

\title[Finite generation of the numerical Grothendieck group]{Finite generation of the \\numerical Grothendieck group}
\author{Gon{\c c}alo~Tabuada}
\address{Gon{\c c}alo Tabuada, Department of Mathematics, MIT, Cambridge, MA 02139, USA}
\email{tabuada@math.mit.edu}
\urladdr{http://math.mit.edu/~tabuada}
\thanks{The author was partially supported by a NSF CAREER Award}

\date{\today}
%
\abstract{Let $k$ be a finite base field. In this note, making use of topological periodic cyclic homology and of the theory of noncommutative motives, we prove that the numerical Grothendieck group of every smooth proper dg $k$-linear category is a finitely generated free abelian group. Along the way, we prove moreover that the category of noncommutative numerical motives over $k$ is abelian semi-simple, as conjectured by Kontsevich. Furthermore, we show that the zeta functions of endomorphisms of noncommutative Chow motives are rational and satisfy a functional equation.}}

\maketitle
\vskip-\baselineskip
\vskip-\baselineskip


\section{Introduction and statement of results}
Given a smooth proper $k$-scheme $X$, over a base field $k$, it is well-known that the (graded) group of algebraic cycles up to numerical equivalence $\cZ^\ast(X)/_{\!\!\sim \mathrm{num}}$ is a finitely generated free abelian group. The proof of this result makes use of the existence of appropriate Weil cohomology theories (\eg\ de Rham cohomology theory in characteristic zero and crystalline cohomology theory in positive characteristic). The main goal of this note is to establish a far-reaching noncommutative generalization of this finite generation result; consult also the motivic results in \S\ref{sec:NCmotives}.

A {\em differential graded (=dg) category} $\cA$, over $k$, is a category enriched over complexes of $k$-vector spaces; see \S\ref{sub:dg}. Every (dg) $k$-algebra $A$ gives naturally rise to a dg category with a single object. Another source of examples is proved by $k$-schemes (or more generally by algebraic $k$-stacks) since the category of perfect complexes $\perf(X)$ of every $k$-scheme $X$ admits a canonical dg enhancement\footnote{When X is quasi-projective this dg enhancement is unique; see Lunts-Orlov \cite[Thm. 2.12]{LO}.} $\perf_\dg(X)$; see Keller \cite[\S4.6]{ICM-Keller}. Following Kontsevich \cite{Miami,finMot,IAS}, a dg category $\cA$ is called {\em smooth} if it is compact as a bimodule over itself and {\em proper} if $\sum_n \mathrm{dim}_k H^n\cA(x,y)<\infty$ for every pair of objects $(x,y)$. Examples include finite dimensional $k$-algebras of finite global dimension $A$ (when $k$ is perfect) as well as the dg categories $\perf_\dg(X)$ associated to smooth proper $k$-schemes $X$ (or to smooth proper algebraic $k$-stacks).

Given a proper dg category $\cA$, its Grothendieck group $K_0(\cA):=K_0(\cD_c(\cA))$ comes equipped with the Euler pairing $\chi \colon  K_0(\cA) \times K_0(\cA) \to \bbZ$ defined as follows:
\begin{equation}\label{eq:pairing}
([M],[N]) \mapsto \sum_n (-1)^n \mathrm{dim}_k \Hom_{\cD_c(\cA)}(M,N[-n])\,.
\end{equation}
This bilinear pairing is, in general, not symmetric neither skew-symmetric. Nevertheless, when $\cA$ is moreover smooth, the associated left and right kernels of $\chi$ agree; see \cite[Prop.~4.24]{book}. Consequently, under these assumptions on $\cA$, we have a well-defined {\em numerical Grothendieck group} $K_0(\cA)/_{\!\!\sim \mathrm{num}}:=K_0(\cA)/\mathrm{Ker}(\chi)$.
\begin{theorem}\label{thm:main1}
When the base field $k$ is finite, the numerical Grothendieck group $K_0(\cA)/_{\!\!\sim \mathrm{num}}$ is a finitely generated free abelian group.
\end{theorem}
The analogue of Theorem \ref{thm:main1}, with $k$ of characteristic zero, was proved in \cite[Thm.~1.2]{Separable}. Therein, we used periodic cyclic homology as a replacement for de Rham cohomology theory. In this note, we use topological periodic cyclic homology (a beautiful theory recently introduced by Hesselholt \cite{Hesselholt}) as a replacement for crystalline cohomology theory.
\begin{remark}[Algebraic cycles up to numerical equivalence]
Let $X$ be a smooth proper $k$-scheme. Thanks to the Hirzebruch-Riemann-Roch theorem, the Chern character yields an isomorphism between $K_0(\perf_\dg(X))/_{\!\!\sim \mathrm{num}} \otimes_\bbZ \bbQ$ and the $\bbQ$-vector space of algebraic cycles up to numerical equivalence $\cZ^\ast(X)/_{\!\!\sim \mathrm{num}}\otimes_\bbZ \bbQ$. Therefore, by applying Theorem \ref{thm:main1} to the dg category $\cA=\perf_\dg(X)$, we conclude that the latter $\bbQ$-vector space is finite dimensional. Since $\cZ^\ast(X)/_{\!\!\sim \mathrm{num}}$ is torsion-free, this implies that $\cZ^\ast(X)/_{\!\!\sim \mathrm{num}}$ is a finitely generated free abelian group.
\end{remark}
\begin{notation}
Throughout the note $k$ denotes a finite base field $\bbF_q$ of order $q=p^r$.
\end{notation}

\section{Background on dg categories}\label{sub:dg}
Let $\cC(k)$ the category of complexes of $k$-vector spaces. A {\em differential
  graded (=dg) category $\cA$} is a category enriched over $\cC(k)$
and a {\em dg functor} $F\colon\cA\to \cB$ is a functor enriched over
$\cC(k)$; consult Keller's ICM survey
\cite{ICM-Keller}. In what follows, we write $\dgcat(k)$ for the category of (essentially small) dg categories.

Let $\cA$ be a dg category. The opposite dg category $\cA^\op$, resp. category $\dgHo(\cA)$, has the same objects as $\cA$ and $\cA^\op(x,y):=\cA(y,x)$, resp. $\dgHo(\cA)(x,y):=H^0\cA(x,y)$. A {\em right dg
  $\cA$-module} is a dg functor $M\colon \cA^\op \to \cC_\dg(k)$ with values
in the dg category $\cC_\dg(k)$ of complexes of $k$-vector spaces. Let
us write $\cC(\cA)$ for the category of right dg
$\cA$-modules. Following \cite[\S3.2]{ICM-Keller}, the derived
category $\cD(\cA)$ of $\cA$ is defined as the localization of
$\cC(\cA)$ with respect to the objectwise quasi-isomorphisms. Let
$\cD_c(\cA)$ be the triangulated subcategory of compact objects.

A dg functor $F\colon\cA\to \cB$ is called a {\em Morita equivalence} if it induces an equivalence on derived categories $\cD(\cA) \simeq
\cD(\cB)$; see \cite[\S4.6]{ICM-Keller}. As explained in
\cite[\S1.6]{book}, the category $\dgcat(k)$ admits a Quillen model
structure whose weak equivalences are the Morita equivalences. Let us
denote by $\Hmo(k)$ the associated homotopy category.

The {\em tensor product $\cA\otimes\cB$} of dg categories is defined
as follows: the set of objects is the cartesian product and
$(\cA\otimes\cB)((x,w),(y,z)):= \cA(x,y) \otimes\cB(w,z)$. As
explained in \cite[\S2.3]{ICM-Keller}, this construction gives rise to
a symmetric monoidal structure $-\otimes -$ on $\dgcat(k)$, which descends to the homotopy category
$\Hmo(k)$. 

Finally, a {\em dg $\cA\text{-}\cB$-bimodule} is a dg functor
$\mathrm{B}\colon \cA\otimes \cB^\op \to \cC_\dg(k)$ or, equivalently, a
right dg $(\cA^\op \otimes \cB)$-module. A standard example is the dg
$\cA\text{-}\cB$-bimodule
\begin{eqnarray}\label{eq:bimodule2}
{}_F\mathrm{B}:\cA\otimes \cB^\op \to \cC_\dg(k) && (x,z) \mapsto \cB(z,F(x))
\end{eqnarray}
associated to a dg functor $F:\cA\to \cB$. Let us write $\rep(\cA,\cB)$ for the full triangulated subcategory of $\cD(\cA^\op \otimes \cB)$ consisting of those dg $\cA\text{-}\cB$-modules $\mathrm{B}$ such that for every object $x \in \cA$ the associated right dg $\cB$-module $\mathrm{B}(x,-)$ belongs to $\cD_c(\cB)$.
\section{Noncommutative motives over a finite field}\label{sec:NCmotives}
For a book, resp. survey, on noncommutative motives, consult \cite{book}, resp. \cite{survey}.
\subsection*{Noncommutative Chow motives}
Let $\cA$ and $\cB$ be two dg categories. As explained in \cite[\S1.6.3]{book}, there is a bijection between $\Hom_{\Hmo(k)}(\cA,\cB)$ and the isomorphism classes of the category $\rep(\cA,\cB)$, under which the composition law of $\Hmo(k)$ corresponds to the derived tensor product of bimodules. Since the dg $\cA\text{-}\cB$ bimodules \eqref{eq:bimodule2} belong to $\rep(\cA,\cB)$, we have the symmetric monoidal functor:
\begin{eqnarray}\label{eq:func1}
\dgcat(k) \too \Hmo(k) & \cA\mapsto \cA & F \mapsto {}_F \mathrm{B}\,.
\end{eqnarray}
The {\em additivization} of $\Hmo(k)$ is the additive category $\Hmo_0(k)$ with the same objects as $\Hmo(k)$ and with morphisms given by $\Hom_{\Hmo_0(k)}(\cA,\cB):=K_0\rep(\cA,\cB)$, where $K_0\rep(\cA,\cB)$ stands for the Grothendieck group of the triangulated category $\rep(\cA,\cB)$. The composition law and the symmetric monoidal structure are induced from $\Hmo(k)$. Consequently, we have the following symmetric monoidal functor:
\begin{eqnarray}\label{eq:func2}
\Hmo(k) \too \Hmo_0(k) & \cA \mapsto \cA & \mathrm{B} \mapsto [\mathrm{B}]\,.
\end{eqnarray}
Given a commutative ring of coefficients $R$, the {\em $R$-linearization} of $\Hmo_0(k)$ is the $R$-linear category $\Hmo_0(k)_R$ obtained by tensoring the morphisms of $\Hmo_0(k)$ with $R$. By construction, we have the following symmetric monoidal functor:
\begin{eqnarray}\label{eq:func3}
\Hmo_0(k) \too \Hmo_0(k)_R & \cA \mapsto \cA & [\mathrm{B}] \mapsto [\mathrm{B}]_R\,.
\end{eqnarray}
Let us denote by $U(-)_R\colon \dgcat(k) \to \Hmo_0(k)_R$ the composition $\eqref{eq:func3}\circ \eqref{eq:func2}\circ \eqref{eq:func1}$.

The category of {\em noncommutative Chow motives} $\NChow(k)_R$ is defined as the idempotent completion of the full subcategory of $\Hmo_0(k)_R$ consisting of the objects $U(\cA)_R$ with $\cA$ a smooth proper dg category. By construction, this latter category is $R$-linear, additive, idempotent complete, and symmetric monoidal. Moreover, since the smooth proper dg categories can be characterized as the dualizable objects of the symmetric monoidal category $\Hmo_0(k)$ (see \cite[Thm.~1.43]{book}), the category $\NChow(k)_R$ is moreover {\em rigid}, \ie all its objects are dualizable. Finally, given dg categories $\cA$ and $\cB$, with $\cA$ smooth proper, we have $\rep(\cA,\cB) \simeq \cD_c(\cA^\op \otimes \cB)$; see \cite[Cor.~1.44]{book}. This implies the following isomorphisms:
\begin{equation}\label{eq:Homs}
\Hom_{\NChow(k)_R}(U(\cA)_R,U(\cB)_R):=K_0(\rep(\cA,\cB))_R\simeq K_0(\cA^\op \otimes\cB)_R\,.
\end{equation}
When $R=\bbZ$, we write $\NChow(k)$ instead of $\NChow(k)_\bbZ$ and $U$ instead of $U(-)_\bbZ$.
\subsection*{Topological periodic cyclic homology}\label{sub:TP}
Let $W(k)$ be the ring of $p$-typical Witt vectors of $k$ and $K:=W(k)_{1/p}$ the fraction field of $W(k)$. Thanks to the work of Hesselholt \cite{Hesselholt}, topological periodic cyclic homology $TP$ (which is defined\footnote{Note that periodic cyclic homology $HP$ may also be defined as the Tate cohomology of the circle group acting on Hochschild homology $HH$.} as the Tate cohomology of the circle group acting on topological Hochschild homology $THH$) yields a lax symmetric monoidal functor 
\begin{equation}\label{eq:TP}
TP_\pm(-)_{1/p} \colon \dgcat(k) \too \mathrm{Vect}_{\bbZ/2}(K)
\end{equation} 
with values in the category of $\bbZ/2$-graded $K$-vector spaces.
\begin{theorem}\label{thm:TP}
Given a ring homomorphism $R \to K$, the lax symmetric monoidal functor \eqref{eq:TP} gives rise to an $R$-linear symmetric monoidal functor
\begin{equation}\label{eq:TP1}
TP_\pm(-)_{1/p}\colon \NChow(k)_R \too \mathrm{vect}_{\bbZ/2}(K)
\end{equation}
with values in the category of finite dimensional $\bbZ/2$-graded $K$-vector spaces.
\end{theorem}
\begin{proof}
Recall from Bondal-Orlov \cite{BO1} that a {\em semi-orthogonal decomposition} of a triangulated category $\cT$, denoted by $\cT=\langle \cT_1, \cT_2\rangle$, consists of full triangulated subcategories $\cT_1, \cT_2 \subseteq \cT$ satisfying the following conditions: the inclusions $\cT_1, \cT_2 \subseteq \cT$ admit left and right adjoints; the triangulated category $\cT$ is generated by the objects of $\cT_1$ and $\cT_2$; and $\Hom_\cT(\cT_2, \cT_1)=0$. A functor $E\colon \dgcat(k)
\to \mathrm{D}$, with values in an additive category, is called an
{\em additive invariant} if it satisfies the conditions:
\begin{itemize}
\item[(i)] It sends the Morita equivalences to isomorphisms.
\item[(ii)] Given dg categories $\cA,\cC \subseteq \cB$ such that $\dgHo(\cB)=\langle\dgHo(\cA), \dgHo(\cC) \rangle$, the inclusions $\cA, \cC\subseteq \cB$ induce an isomorphism $E(\cA) \oplus E(\cC) \simeq E(\cB)$.
\end{itemize}
As explained in \cite[\S2.3]{book}, the functor $U(-)_R\colon \dgcat(k) \to \Hmo_0(k)_R$ is the {\em universal} additive invariant, \ie given any $R$-linear, additive, and symmetric monoidal category $\mathrm{D}$, we have an induced equivalence of categories
\begin{equation}\label{eq:equivalence}
U(-)_R^\ast \colon \mathrm{Fun}^\otimes_{R\text{-}\mathrm{linear}}(\Hmo_0(k)_R, \mathrm{D}) \stackrel{\simeq}{\too} \mathrm{Fun}^\otimes_{\mathrm{add}}(\dgcat(k), \mathrm{D})\,,
\end{equation}
where the left-hand side denotes the category of $R$-linear lax symmetric monoidal functors and the right-hand side the category of lax symmetric monoidal additive invariants. Since topological Hochschild homology $THH$ is a (lax) symmetric monoidal additive invariant (see \cite[\S2.2.12]{book}), it follows from the exactness of the Tate construction that the above functor \eqref{eq:TP} is also a lax symmetric monoidal additive invariant. Consequently, making use of the ring homomorphism $R \to K$ and of the equivalence of categories \eqref{eq:equivalence}, we conclude that \eqref{eq:TP} gives rise to an $R$-linear lax symmetric monoidal functor 
\begin{equation}\label{eq:TP2}
TP_\pm(-)_{1/p} \colon \Hmo_0(k)_R \too \mathrm{Vect}_{\bbZ/2}(K)\,.
\end{equation}
Given smooth proper dg categories $\cA$ and $\cB$, Blumberg-Mandell \cite{BM} proved that the natural morphism of $TP(k)$-modules $TP(\cA) \wedge_{TP(k)}TP(\cB) \to TP(\cA\otimes \cB)$ is a weak equivalence. Consequently, the induced morphism of $TP(k)_{1/p}$-modules
\begin{equation}\label{eq:TP5}
TP(\cA)_{1/p} \wedge_{TP(k)_{1/p}}TP(\cB)_{1/p} \too TP(\cA\otimes \cB)_{1/p}
\end{equation}
is also a weak equivalence. Making use of the (convergent) Tor spectral sequence
$$ E^2_{i,j} = \mathrm{Tor}_{i,j}^{TP_\ast(k)_{1/p}}(TP_\ast(\cA)_{1/p}, TP_\ast(\cB)_{1/p}) \Rightarrow \mathrm{Tor}_{i+j}^{TP(k)_{1/p}} (TP(\cA)_{1/p}, TP(\cB)_{1/p})$$
and of the computation $TP_\ast(k)_{1/p} \simeq \mathrm{Sym}_K\{v^{\pm 1}\}$, where the divided Bott element $v$ is of degree $-2$ (see \cite[\S4]{Hesselholt}), we hence conclude that \eqref{eq:TP5} yields an isomorphism of $\bbZ/2$-graded $K$-vector spaces $TP_\pm(\cA)_{1/p} \otimes TP_\pm(\cB)_{1/p} \simeq TP_\pm(\cA\otimes \cB)_{1/p}$. By definition of the category of noncommutative Chow motives, this implies that \eqref{eq:TP2} gives rise to an $R$-linear symmetric monoidal functor
\begin{equation}\label{eq:TP6}
TP_\pm(-)_{1/p}\colon \NChow(k)_R \too \mathrm{Vect}_{\bbZ/2}(K)\,.
\end{equation}
Finally, since the category $\NChow(k)_R$ is rigid and the rigid objects of $\mathrm{Vect}_{\bbZ/2}(K)$ are the finite dimensional $\bbZ/2$-graded $K$-vector spaces, the preceding symmetric monoidal functor \eqref{eq:TP6} takes values in $\mathrm{vect}_{\bbZ/2}(K)$. This concludes the proof.
\end{proof}
Let $K_0(\NChow(k)_R)$ and $K_0(\mathrm{vect}_{\bbZ/2}(K))$ be the Grothendieck rings of the additive symmetric monoidal categories $\NChow(k)_R$ and $\mathrm{vect}_{\bbZ/2}(K)$, respectively. Note that we have the following canonical isomorphism:
\begin{eqnarray*}\label{eq:iso}
K_0(\mathrm{vect}_{\bbZ/2}(K))\stackrel{\simeq}{\too} \bbZ[\epsilon]/(\epsilon^2=1) && [(M_+, M_-)]\mapsto \mathrm{dim}_K(M_+) - \mathrm{dim}_K (M_-)\epsilon\,.
\end{eqnarray*}
Since the functor \eqref{eq:TP1} is symmetric monoidal, it induces a ring homomorphism
$$ TP_\pm(-)_{1/p}\colon K_0(\NChow(k)_R) \too K_0(\mathrm{vect}_{\bbZ/2}(K))\simeq \bbZ[\epsilon]/(\epsilon^2 = 1)\,.$$
\begin{proposition}\label{prop:crystalline}
Given a smooth proper $k$-scheme $X$, we have the equality
\begin{equation}\label{eq:TP_new}
TP_\pm([U(\perf_\dg(X))_R])_{1/p} = \sum_n (-1)^n \mathrm{dim}_K (H^n_{\mathrm{crys}}(X)_{1/p})\epsilon^{n/2}\,,
\end{equation}
where $H^\ast_{\mathrm{crys}}(X)$ stands for the crystalline cohomology of $X$.
\end{proposition}
Roughly speaking, Proposition \ref{prop:crystalline} shows that ``virtually'' the topological periodic cyclic homology of $X$ agrees with the crystalline cohomology of $X$.
\begin{proof}
On the one hand, as proved by Hesselholt in \cite[Thm.~5.1]{Hesselholt}, we have the (convergent) conjugate spectral sequence
\begin{equation}\label{eq:spectral1}
 E_2^{i,j} = \mathrm{lim}_{n, \mathrm{Frob}} \,H^i(X, W_n\, \Omega_X^j)_{1/p} \Rightarrow H^{i+j}_{\mathrm{crys}}(X)_{1/p}\,,
 \end{equation}
in which the $E_2$-terms are given by hypercohomology with coefficients in the de Rham-Witt complex. On the other hand, as proved once again by Hesselholt in \cite[Thm.~6.8]{Hesselholt}, we have the (convergent) Hodge spectral sequence
\begin{equation}\label{eq:spectral2}
E^2_{i,j} = \bigoplus_{m \in \bbZ} \mathrm{lim}_{n, \mathrm{Frob}}\,H^{-i}(X, W_n\, \Omega_X^{j+2m})_{1/p} \Rightarrow TP_{i+j}(\perf_\dg(X))_{1/p}\,,
\end{equation}
in which the differentials preserve the direct sum decomposition of the $E^2$-term. This implies that the dimension of the $K$-vector space $TP_+(\perf_\dg(X))_{1/p}$, resp. $TP_-(\perf_\dg(X))_{1/p}$, is equal to the (finite) sum $\sum_{n \,\mathrm{even}} \mathrm{dim}_K (H^n_{\mathrm{crys}}(X)_{1/p})$, resp. $\sum_{n \,\mathrm{odd}} \mathrm{dim}_K (H^n_{\mathrm{crys}}(X)_{1/p})$. Consequently, we obtain the above equality \eqref{eq:TP_new}.
%
%
\end{proof}
\subsection*{Noncommutative numerical motives}
Given an $R$-linear, additive, rigid symmetric monoidal category $(\mathrm{D}, \otimes, {\bf 1})$, its {\em $\cN$-ideal} is defined as follows
$$ \cN(a,b):=\{f \in \Hom_{\mathrm{D}}(a,b)\,|\, \forall g \in \Hom_{\mathrm{D}}(b,a)\,\,\mathrm{we}\,\,\mathrm{have}\,\,\mathrm{tr}(g\circ f)=0 \}\,,$$
where $\mathrm{tr}(g\circ f)$ stands for the categorical trace of the endomorphism $g\circ f$. The category of {\em noncommutative numerical motives} $\NNum(k)_R$ is defined as the idempotent completion of the quotient of $\NChow(k)_R$ by the $\otimes$-ideal $\cN$. By construction, this category  is $R$-linear, additive, idempotent complete, and rigid symmetric monoidal. 
\begin{theorem}[Semi-simplicity]\label{thm:semi}
When $R$ is a field of characteristic zero, the category of noncommutative numerical motives $\NNum(k)_R$ is abelian semi-simple.
\end{theorem}
Assuming certain (polarization) conjectures, Kontsevich conjectured in his seminal talk \cite{IAS} that the category of noncommutative numerical motives $\NNum(k)_R$ was abelian semi-simple. Theorem \ref{thm:semi} shows that Kontsevich's insight holds unconditionally. The analogue of Theorem \ref{thm:semi}, with $k$ and $R$ fields of the same characteristic, was proved in \cite[Thm.~5.6]{JEMS} and \cite[Thm.~1.10]{AJM}.
\begin{proof}
Let $R'/R$ be a field extension.  As proved in \cite[Lem.~4.29]{book}, if the category $\NNum(k)_R$ is abelian semi-simple, then the category $\NNum(k)_{R'}$ is also semi-simple. Hence, we can assume without loss of generality that $R=\bbQ$. Since the fraction field $K:=W(k)_{1/p}$ is of characteristic zero, Theorem \ref{thm:TP} provides us with a $\bbQ$-linear symmetric monoidal functor 
\begin{equation}\label{eq:TP7}
TP_\pm(-)_{1/p} \colon \NChow(k)_\bbQ \too \mathrm{vect}_{\bbZ/2}(K)\,.
\end{equation}
Clearly, the $K$-linear category $\mathrm{vect}_{\bbZ/2}(K)$ is rigid symmetric monoidal and satisfies conditions (i)-(ii) of Theorem \ref{thm:AK} below. Moreover, following \eqref{eq:Homs}, we have an isomorphism $\mathrm{End}_{\NChow(k)_\bbQ}(U(k)_\bbQ)\simeq \bbQ$. Therefore, making use of Theorem \ref{thm:AK} below (with $H$ given by the symmetric monoidal functor \eqref{eq:TP7}), we conclude that the category $\NNum(k)_\bbQ$ is abelian semi-simple. 
\end{proof}
\begin{theorem}{(Andr\'e-Kahn \cite[Thm.~1]{AK1})}\label{thm:AK}
Let $R$ be a field and $(\mathrm{D},\otimes, {\bf 1})$ an $R$-linear, additive, rigid symmetric monoidal category with $\mathrm{End}_{\mathrm{D}}({\bf 1})\simeq R$. Assume that there exists a symmetric monoidal functor $H\colon \mathrm{D} \to \mathrm{D}'$ with values in a $R'$-linear rigid symmetric monoidal category (with $R'/R$ a field extension) such that:
\begin{itemize}
\item[(i)] We have $\mathrm{dim}_{R'} \Hom_{\mathrm{D}'}(a,b)< \infty$ for any two objects $a$ and $b$.
\item[(ii)] Every nilpotent endomorphism in $\mathrm{D}'$ has a trivial categorical trace.
\end{itemize}   
Under these assumptions, the idempotent completion of the quotient of $\mathrm{D}$ by the $\otimes$-ideal $\cN$ is an abelian semi-simple category.
\end{theorem}
\begin{remark}[Semi-simplicity of the category of numerical motives]
Let $R$ be a field of characteristic zero and $\Chow(k)_R$ and $\Num(k)_R$ the (classical) categories of Chow motives and numerical motives, respectively. Making use of \'etale cohomology, Jannsen proved in \cite{Jannsen} that the category $\Num(k)_R$ is abelian semi-simple, thus solving a conjecture of Grothendieck. Here is an alternative proof of this important result: as explained in \cite[\S4.2]{book}, there exists a $\bbQ$-linear, fully-faithful, symmetric monoidal functor $\Phi$ making the following diagram commute
\begin{equation}\label{eq:bridge}
\xymatrix@C=2.5em@R=1.5em{
\mathrm{SmProp}(k)^\op \ar[rr]^-{X \mapsto \perf_\dg(X)} \ar[d]_-{\mathfrak{h}(-)_\bbQ} && \dgcat_{\mathrm{sp}}(k) \ar[dd]^-{U(-)_\bbQ} \\
\Chow(k)_\bbQ \ar[d]_-\pi & & \\
\Chow(k)_\bbQ/_{\!\!- \otimes \bbQ(1)} \ar[rr]_-\Phi && \NChow(k)_\bbQ \,,
}
\end{equation}
where $\mathrm{SmProp}(k)$, resp. $\dgcat_{\mathrm{sp}}(k)$, stands for the category of smooth proper $k$-schemes, resp. for the category of smooth proper dg categories, and $\Chow(k)_\bbQ/_{\!\!-\otimes \bbQ(1)}$ for the orbit category with respect to the Tate motive $\bbQ(1)$. By applying the above Theorem \ref{thm:AK} to the following composition
$$ \Chow(k)_\bbQ \stackrel{\pi}{\too} \Chow(k)_\bbQ/_{\!\!- \otimes \bbQ(1)} \stackrel{\Phi}{\too} \NChow(k)_\bbQ \stackrel{\eqref{eq:TP7}}{\too} \mathrm{vect}_{\bbZ/2}(K)\,,$$
we hence conclude that the category $\Num(k)_\bbQ$ (defined as the idempotent completion of the quotient $\Chow(k)_\bbQ/\cN$) is abelian semi-simple. Similarly to the proof of Theorem \ref{thm:semi}, the general case follows now from \cite[Lem.~4.29]{book}.
\end{remark}
\begin{remark}[Noncommutative motivic Galois group]
Similarly to \cite[Chap.~6]{book}, making use of Theorems \ref{thm:TP} and \ref{thm:semi} and of the Tannakian (super-)formalism, we can attach to the abelian semi-simple category of noncommutative numerical motives a noncommutative motivic Galois (super-)group. The study of the properties of this affine group (super-)scheme is the subject of current research.
\end{remark}
\begin{remark}[Jacobian]
Similarly to \cite[Chap.~7]{book}, making use of Theorem \ref{thm:semi} and of the extension of the above bridge \eqref{eq:bridge} to numerical motives (see \cite[\S4.6]{book}), we can construct a ``Jacobian'' functor ${\bf J}(-)\colon \dgcat_{\mathrm{sp}}(k) \to \mathrm{Ab}(k)_\bbQ$ with values in the category of abelian varieties up to isogeny. In the case of a smooth proper $k$-curve $C$ and smooth proper $k$-surface $S$, we have ${\bf J}(\perf_\dg(C))\simeq J(C)$ and ${\bf J}(\perf_\dg(S))\simeq \mathrm{Pic}^0(S) \times \mathrm{Alb}(S)$, respectively. The study of the properties of the ``Jacobian'' functor ${\bf J}(-)$ is the subject of current research.
\end{remark}
\subsection*{Zeta functions of endomorphisms}
Let $N\!\!M \in \NChow(k)_\bbQ$ be a noncommutative Chow motive and $f\in \mathrm{End}_{\NChow(k)_\bbQ}(N\!\!M)$ an endomorphism. Following Kahn \cite[Def.~3.1]{Zeta}, the {\em zeta function of $f$} is defined as the following formal power series
\begin{equation}\label{eq:zeta}
Z(f;t):= \mathrm{exp}(\sum_{n \geq 1} \mathrm{tr}(f^{\circ n}) \frac{t^n}{n}) \in \bbQ\llbracket t \rrbracket\,,
\end{equation}
where $f^{\circ n}$ stands for the composition of $f$ with itself $n$-times, $\mathrm{tr}(f^{\circ n}) \in \bbQ$ for the categorical trace\footnote{Let $(\mathrm{D}, \otimes, {\bf 1})$ be a rigid symmetric monoidal category. Given an endomorphism $f\colon a \to a$, recall that its {\em categorical trace $\mathrm{tr}(f)$} is defined as ${\bf 1} \stackrel{\mathrm{co}}{\to} a^\vee \otimes a \simeq a \otimes a^\vee \stackrel{f\otimes \id}{\to} a\otimes a^\vee \stackrel{\mathrm{ev}}{\to} {\bf 1}$, where $\mathrm{co}$ stands for the co-evaluation morphism and $\mathrm{ev}$ for the evaluation morphism.} of $f^{\circ n}$, and $\mathrm{exp}(t):=\sum_{m \geq 0} \frac{t^m}{m!} \in \bbQ\llbracket t\rrbracket$. When $N\!\!M = U(\cA)_\bbQ$, with $\cA$ a smooth proper dg category, and $f=[\mathrm{B}]_\bbQ$, with $\mathrm{B} \in \cD_c(\cA^\op \otimes \cA)$ a dg $\cA\text{-}\cA$-bimodule, we have the following computation (see \cite[Prop.~2.26]{book})
\begin{equation}\label{eq:integers}
\mathrm{tr}(f^{\circ n}) =[HH(\cA, \underbrace{\mathrm{B}\otimes^{\bf L}_\cA \cdots \otimes^{\bf L}_\cA \mathrm{B}}_{n\text{-}\text{times}})] \in K_0(k) \simeq \bbZ\,,
\end{equation}
where $HH(\cA, \mathrm{B}\otimes^{\bf L}_\cA \cdots \otimes^{\bf L}_\cA \mathrm{B})$ stands for the Hochschild homology of $\cA$ with coefficients in $\mathrm{B}\otimes^{\bf L}_\cA \cdots \otimes^{\bf L}_\cA \mathrm{B}$. Intuitively speaking, the integer \eqref{eq:integers} is the ``number of fixed points'' of the dg $\cA\text{-}\cA$-bimodule $\mathrm{B}\otimes^{\bf L}_\cA \cdots \otimes^{\bf L}_\cA \mathrm{B}$.
\begin{example}[Frobenius]
When $N\!\!M= U(\perf_\dg(X))_\bbQ$, with $X$ a smooth proper $k$-scheme, and $f=[{}_{\mathrm{Fr}^\ast}\mathrm{B}]_\bbQ$, with $\mathrm{Fr}^\ast\colon \perf_\dg(X) \to \perf_\dg(X)$ the pull-back dg functor associated to the Frobenius $\mathrm{Fr}$ of $X$, we have $\mathrm{tr}(f^{\circ n}) = \langle \Delta \cdot \Gamma_{\mathrm{Fr}^n}\rangle=~|X(\bbF_{q^n})|$. Consequently, we conclude that $Z(f;t)= Z_X(t):= \prod_{x \in X} (1-t^{\mathrm{deg}(x)})^{-1}$, where the product is taken over the closed points $x$ of $X$.
\end{example}
\begin{remark}[Witt vectors]
Recall from \cite{Hazewinkel} the definition of the ring of (big) Witt vectors $\mathrm{W}(\bbQ)=(1 + t \bbQ\llbracket t \rrbracket, \times, \ast)$. Since the leading term of \eqref{eq:zeta} is equal to $1$, the zeta function $Z(f;t)$ belongs to $\mathrm{W}(\bbQ)$. Moreover, given endomorphisms $f \in \mathrm{End}(N\!\!M)$ and $f' \in \mathrm{End}(N\!\!M')$, we have $Z(f\oplus f';t)=Z(f;t) \times Z(f';t)$ and $Z(f\otimes f; t) = Z(f;t) \ast Z(f';t)$ in the ring of Witt vectors $\mathrm{W}(\bbQ)$.
\end{remark}
Let $B = \prod_i B_i$ be a finite dimensional semi-simple $\bbQ$-algebra. Following Kahn \cite[\S1]{Zeta}, let us write $Z_i$ for the center of $B_i$, $\delta_i$ for the degree $[Z_i:\bbQ]$, and $d_i$ for the index $[B_i: Z_i]^{1/2}$. Given a unit $b \in B^\times$, the {\em $i^{\mathrm{th}}$ reduced norm} $\mathrm{Nrd}_i(b) \in \bbQ$ is defined as the composition $(\mathrm{N}_{Z_i/\bbQ}\circ \mathrm{Nrd}_{B_i/Z_i})(b_i)$. Thanks to Theorem \ref{thm:semi}, $B:=\mathrm{End}_{\NNum(k)_\bbQ}(N\!\!M)$ is a finite dimensional semi-simple $\bbQ$-algebra; let us write $e_i \in B$ for the central idempotent corresponding to the summand $B_i$. Given an invertible endomorphism $f \in \End_{\NChow(k)_\bbQ}(N\!\!M)^\times$, its {\em determinant $\mathrm{det}(f) \in \bbQ$} is defined as the product $\prod_i \mathrm{Nrd}_i(f)^{\mu_i}$, where $\mu_i \in \bbZ$ stands for the quotient $\frac{\mathrm{tr}(e_i)}{\delta_i d_i}$.
\begin{theorem}\label{thm:zeta}
\begin{itemize}
\item[(i)] Given an embedding $\iota\colon K \hookrightarrow \bbC$, we have the following equality of formal power series with $\bbC$-coefficients: 
\begin{equation}\label{eq:zeta1}
 Z(f;t) = \frac{\mathrm{det}(\id - t \,TP_-(f)_{1/p}\,|\, TP_-(N\!\!M)_{1/p} \otimes_{K, \iota} \bbC)}{\mathrm{det}(\id - t\, TP_+(f)_{1/p}\,|\, TP_+(N\!\!M)_{1/p} \otimes_{K, \iota} \bbC)}\,.
\end{equation}
\item[(ii)] The formal power series $Z(f;t) \in \bbQ \llbracket t \rrbracket$ is {\em rational}, \ie it can be written as $\frac{p(t)}{q(t)}$ with $p(t), q(t) \in \bbQ[t]$. Moreover, we have $\mathrm{deg}(q(t)) - \mathrm{deg}(p(t))= \mathrm{tr}(\id_{N\!\!M})$.  
\item[(iii)] When $f$ is invertible, we have the following functional equation:
$$ Z(f^{-1};t^{-1}) = (-t)^{\mathrm{tr}(\id_{N\!\!M})} \mathrm{det}(f) Z(f;t)\,.$$
\end{itemize}
\end{theorem}
\begin{proof}
Thanks to Theorem \ref{thm:TP}, we have the following equality in $\bbC\llbracket t \rrbracket$:
\begin{equation}\label{eq:equality-verylast}
Z(f;t)= \mathrm{exp}(\sum_{n \geq 1} (\mathrm{tr}(TP_+(f^{\circ n})_{1/p})-\mathrm{tr}(TP_-(f^{\circ n})_{1/p}))\frac{t^n}{n}) \,.
\end{equation}
Moreover, the right-hand side of \eqref{eq:equality-verylast} agrees with
$$ \frac{\mathrm{exp}(\sum_{n \geq 1} \mathrm{tr}(TP_+(f^{\circ n})_{1/p})\frac{t^n}{n})}{\mathrm{exp}(\sum_{n \geq 1} \mathrm{tr}(TP_-(f^{\circ n})_{1/p})\frac{t^n}{n})} = \frac{\mathrm{exp}(\mathrm{log}(\mathrm{det}(\id - t \,TP_+(f)_{1/p})^{-1}))}{\mathrm{exp}(\mathrm{log}(\mathrm{det}(\id - t \,TP_-(f)_{1/p})^{-1}))} \,,$$
where the latter equality follows from \cite[Lem.~22.5]{Milne} (since $\bbC$ is algebraically closed). This proves item (i). Since $Z(f;t) \in \bbQ\llbracket t \rrbracket$ is rational as a formal power series with $\bbC$-coefficients, it is also rational as a formal power series with $\bbQ$-coefficients; see \cite[Lem.~27.9]{Milne}. Moreover, we have the following equality
$$
\mathrm{deg}(q(t))-\mathrm{deg}(p(t)) = \mathrm{deg}(\mathrm{det}(\id - t \,TP_+(f)_{1/p})) - \mathrm{deg}(\mathrm{det}(\id - t \,TP_-(f)_{1/p}))\,.
$$
This implies item (ii) because, thanks to Theorem \ref{thm:TP}, the right side agrees with
$$ \mathrm{dim}_\bbC (TP_+(N\!\!M)_{1/p} \otimes_{K, \iota} \bbC) -  \mathrm{dim}_\bbC (TP_-(N\!\!M)_{1/p} \otimes_{K, \iota} \bbC) = \mathrm{tr}(\id_{N\!\!M})\,.$$
Finally, note that $Z(f;t)$ can be computed in $\NChow(k)_\bbQ$ or in the abelian semi-simple category $\NNum(k)_\bbQ$. Therefore, item (iii) follows from \cite[Thm.~3.2 b)]{Zeta}.
\end{proof}
Given a smooth proper $k$-scheme $X$, recall that its Hasse-Weil zeta function is defined as $\zeta_X(s) := \prod_{x \in X} (1- (q^{-s})^{\mathrm{deg}(x)})^{-1}$. This product converges and defines $\zeta_X(s)$ as an holomorphic function for every $s \in \bbC$ such that $\mathrm{Re}(s)> \mathrm{dim}(X)$.
\begin{corollary}\label{cor:zeta}
Let $X$ be a irreducible smooth proper $k$-scheme $X$ of dimension $d$, and $\mathrm{E}:=\langle \Delta \cdot \Delta \rangle \in \bbZ$ the self-intersection number of the diagonal $\Delta$ of $X \times X$. 
\begin{itemize}
\item[(i)] Given an embedding $\iota\colon K \hookrightarrow \bbC$, we have the equality of holomorphic functions:
\begin{equation}\label{eq:zeta2}
\zeta_X(s) = \frac{\mathrm{det}(\id - q^{-s} \,TP_-(\mathrm{Fr}^\ast)_{1/p}\,|\, TP_-(\perf_\dg(X))_{1/p} \otimes_{K, \iota} \bbC)}{\mathrm{det}(\id - q^{-s}\, TP_+(\mathrm{Fr}^\ast)_{1/p}\,|\, TP_+(\perf_\dg(X))_{1/p} \otimes_{K, \iota} \bbC)}\,.
\end{equation}
\item[(ii)] The formal power series $Z_X(t)\in \bbQ\llbracket t \rrbracket$ is rational, \ie it can be written as $\frac{p(t)}{q(t)}$ with $p(t), q(t) \in \bbQ[t]$. Moreover, we have $\mathrm{deg}(q(t))-\mathrm{deg}(p(t))= \mathrm{E}$.
\item[(iii)] We have the functional equation $Z_X(\frac{1}{q^d t}) = \pm t^{\mathrm{E}} q^{\frac{d}{2} \mathrm{E}} Z_X(t)$.
\end{itemize}
\end{corollary}
\begin{proof}
Thanks to Theorem \ref{thm:zeta}, item (i) follows from the equality $\zeta_X(s)= Z_X(q^{-s})$, item (ii) from the computation $\mathrm{tr}(\id_{U(\perf_\dg(X))_\bbQ})=\mathrm{E}$, and item (iii) from the computation $\mathrm{det}([{}_{\mathrm{Fr}^\ast} \mathrm{B}]_\bbQ)=\pm q^{\frac{d}{2} \mathrm{E}}$ (see \cite[App.~C Thm.~4.4]{Hartshorne}).
\end{proof}
The equality \eqref{eq:zeta2} provides an interpretation of the Hasse-Weil zeta function in terms of topological periodic cyclic homology. This interpretation was originally established by Hesselholt in \cite{Hesselholt}. His proof is based on the combination of the spectral sequences \eqref{eq:spectral1}-\eqref{eq:spectral2} with Berthelot's interpretation \cite[\S VIII]{Berthelot} of the Hasse-Weil zeta function in terms of crystalline cohomology. Note that the above equality \eqref{eq:zeta1} provides a far-reaching noncommutative generalization of this phenomenon. 
\begin{remark}[Analytic continuation]
It is well-known that the Hasse-Weil zeta function admits an analytic continuation to the entire complex plane as a meromorphic function. Making use of Deninger's theory of regularized determinants, Hesselholt was able to remarkably modify the right-hand side of \eqref{eq:zeta2} in order to obtain an equality of meromorphic functions!; see \cite[Thm.~A]{Hesselholt}.
\end{remark}
\begin{remark}[Weil conjecture]
The celebrated Weil conjecture \cite{Weil} is divided into four claims about the Hasse-Weil zeta function: rationality, functional equation, analogue of the Riemann hypothesis, and Betti numbers; consult \cite[App.~C]{Hartshorne}. Item (ii), resp. item (iii), of Corollary \ref{cor:zeta} provides a solution to the rationality, resp. functional equation, claim of the Weil conjecture. Dwork's, resp. Grothendieck's, solution of these claims used $p$-adic analysis, resp. \'etale cohomology; see \cite{Dwork, Grothendieck}.

The above considerations suggest that item (ii), resp. item (iii), of Theorem \ref{thm:zeta} may be understood as the solution of the rational, resp. functional equation, claim of a far-reaching ``noncommutative Weil conjecture''.
\end{remark}
\section{Proof of Theorem \ref{thm:main1}}
Note first that, by definition, the numerical Grothendieck group $K_0(\cA)/_{\!\!\sim \mathrm{num}}$ is torsion-free. Therefore, in order to prove that $K_0(\cA)/_{\!\!\sim \mathrm{num}}$ is a finitely generated free abelian group, it suffices to show that the $\bbQ$-vector space $K_0(\cA)/_{\!\!\sim \mathrm{num}}\otimes_\bbZ \bbQ$ is finite dimensional. Consider the following bilinear pairing:
\begin{eqnarray*}
\varphi\colon K_0(\cA) \times K_0(\cA^\op) \too K_0(k) \simeq \bbZ && ([M],[P]) \mapsto [M\otimes^{\bf L}_\cA P]\,.
\end{eqnarray*}
Under \eqref{eq:Homs}, $\varphi$ corresponds to the composition bilinear pairing in $\NChow(k)$:
$$
\Hom(U(k),U(\cA))\times \Hom(U(\cA),U(k))\too \mathrm{End}(U(k)) \quad (f,g) \mapsto g\circ f = \mathrm{tr}(g\circ f)\,.
$$
Let $M\in \cD_c(\cA)$. Thanks to Proposition \ref{prop:key2} below, we have $\chi([M],[N])=0$ for every $N \in \cD_c(\cA)$ if and only if $\varphi([M],[P])=0$ for every $P \in \cD_c(\cA^\op)$. This implies that $K_0(\cA)/_{\!\!\sim \mathrm{num}}$ is isomorphic to $\Hom_{\NNum(k)}(U(k),U(\cA))$. Consequently, we conclude that the $\bbQ$-vector space $K_0(\cA)/_{\!\!\sim \mathrm{num}} \otimes_\bbZ \bbQ$ is isomorphic to
\begin{equation}\label{eq:vector}
\Hom_{\NNum(k)}(U(k),U(\cA))\otimes_\bbZ \bbQ \simeq \Hom_{\NNum(k)_\bbQ}(U(k)_\bbQ, U(\cA)_\bbQ)\,,
\end{equation}
where the latter isomorphism follows from the compatibility of the $\otimes$-ideal $\cN$ with change of coefficients. Finally, since the category $\NNum(k)_\bbQ$ is abelian semi-simple, the preceding $\bbQ$-vector space \eqref{eq:vector} is finite dimensional. This concludes the proof.
\begin{proposition}\label{prop:key2}
Given a right dg $\cA$-module $M \in \cD_c(\cA)$, we have $\chi([M],[N])=0$ for every $N \in \cD_c(\cA)$ if and only if $\varphi([M],[P])=0$ for every $P \in \cD_c(\cA^\op)$.
\end{proposition}
\begin{proof}
Note first that the differential graded structure of the dg category of complexes of $k$-vector spaces $\cC_\dg(k)$ makes the category of right dg $\cA$-modules $\cC(\cA)$ into a dg category $\cC_\dg(\cA)$. Given right dg $\cA$-modules $M,N \in \cD_c(\cA)$, let us write ${\bf R}\Hom_\cA(M,N) \in \cD_c(k)$ for the corresponding (derived) complex of $k$-vector spaces. Since the Grothendieck class $[{\bf R}\Hom_\cA(M,N)] \in K_0(k)\simeq \bbZ$ agrees with the alternating sum $\sum_n (-1)^n \mathrm{dim}_K\, \Hom_{\cD_c(\cA)}(M,N[-n])$, the Euler bilinear pairing $\chi\colon K_0(\cA) \times K_0(\cA) \to \bbZ$ can be re-written as $([M],[N]) \mapsto [{\bf R}\Hom_\cA(M,N)]$.

Given a right dg $\cA$-module $M$, let us denote by $M^\ast$ the right dg $\cA^\op$-module obtained by composing the dg functor $M\colon \cA^\op \to \cC_\dg(k)$ with the $k$-linear duality $(-)^\ast$. Note that since the dg category $\cA$ is proper, the assignment $M\mapsto M^\ast$ gives rise to an equivalence of categories $\cD_c(\cA) \simeq \cD_c(\cA^\op)^\op$. 

We now have the ingredients necessary to conclude the proof. Assume that $\chi([M],[N])=0$ for every $N \in \cD_c(\cA)$. We need to show that $\varphi([M],[P])=0$ for every $P \in \cD_c(\cA^\op)$. In other words, we need to show that the Grothendieck class $[M\otimes^{\bf L}_\cA P] \in K_0(k)\simeq \bbZ$ is zero for every $P \in \cD_c(\cA^\op)$. Thanks to Lemma \ref{lem:key}(ii) below, the (derived) complex of $k$-vector spaces $M\otimes^{\bf L}_\cA P$ is $k$-linear dual to ${\bf R}\Hom_\cA(M,P^\ast)$. Consequently, since $[{\bf R}\Hom_\cA(M,P^\ast)]=\chi([M],[P^\ast])=0$, we conclude that $[M\otimes^{\bf L}_\cA P]=0$. The proof of the converse implication is similar.
\end{proof}
\begin{lemma}\label{lem:key}
\begin{itemize}
\item[(i)] Given $M,N \in \cD_c(\cA)$, the associated (derived) complexes of $k$-vector spaces ${\bf R}\Hom_\cA(M,N)$ and $M\otimes^{\bf L}_\cA N^\ast$ are $k$-linear dual to each other. 
\item[(ii)] Given $M\in \cD_c(\cA)$ and $P \in \cD_c(\cA^\op)$, the associated (derived) complexes of $k$-vector spaces ${\bf R}\Hom_\cA(M,P^\ast)$ and $M\otimes^{\bf L}_\cA P$ are $k$-linear dual to each other. 
\end{itemize}
\end{lemma}
\begin{proof}
We prove solely item (i); the proof of item (ii) is similar. Given an object $y \in \cA$ and an integer $n \in \bbZ$, consider the following right dg $\cA$-module:
\begin{eqnarray*}
\widehat{y}[n] \colon \cA^\op \too \cC_\dg(k) && x \mapsto \cA(x,y)[n]\,.
\end{eqnarray*}
Every right dg $\cA$-module $M \in \cD_c(\cA)$ is a retract of a finite homotopy colimit of right dg $\cA$-modules of the form $\widehat{y}[n]$. When $M=\widehat{y}[n]$, we have:
$$ {\bf R}\Hom_\cA(\widehat{y}[n], N)^\ast \simeq (N(y)[-n])^\ast \simeq N^\ast(y)[n]\simeq \widehat{y}[n]\otimes^{\bf L}_\cA N^\ast\,.$$
In the same vein, when $M=\mathrm{hocolim}_{i\in I} \,\widehat{y_i}[n_i]$, we have natural isomorphisms
\begin{eqnarray}
{\bf R}\Hom_\cA(\mathrm{hocolim}_{i\in I} \,\widehat{y_i}[n_i], N)^\ast & \simeq & (\mathrm{holim}_{i\in I}  \,{\bf R}\Hom_\cA(\widehat{y_i}[n_i],N))^\ast \nonumber \\
& \simeq & \mathrm{hocolim}_{i \in I} \,{\bf R}\Hom_\cA(\widehat{y_i}[n_i],N)^\ast \label{eq:star-last}\\
& \simeq & \mathrm{hocolim}_{i \in I} \,(\widehat{y_i}[n_i]\otimes^{\bf L}_\cA N^\ast)  \nonumber\\
& \simeq & (\mathrm{hocolim}_{i \in I}\, \widehat{y_i}[n_i])\otimes^{\bf L}_\cA N^\ast\,, \nonumber
\end{eqnarray}
where in \eqref{eq:star-last} we are using the fact that the indexing category $I$ is finite. The proof follows now from the fact that the $k$-linear duality $M \mapsto M^\ast$ preserves retracts.
\end{proof}

\begin{remark}[Rank]
As the proof of Theorem \ref{thm:main1} shows, the rank of the numerical Grothendieck group $K_0(\cA)/_{\!\!\sim \mathrm{num}}$ is $\leq \mathrm{dim}_K TP_+(\cA)_{1/p}$.
\end{remark}
\begin{remark}[Kontsevich's noncommutative numerical motives]
As explained in \S\ref{sec:NCmotives}, the category $\NChow(k)$ is rigid. Therefore, as the proof of Theorem \ref{thm:main1} shows, we have $\Hom_{\NNum(k)}(U(\cA),U(\cB)) \simeq K_0(\cA^\op \otimes \cB)/\mathrm{Ker}(\chi)$ for any two smooth proper dg categories $\cA$ and $\cB$. This implies that we can alternatively define the category of noncommutative numerical motives $\NNum(k)$ as the idempotent completion of the quotient of $\NChow(k)$ by the $\otimes$-ideal $\mathrm{Ker}(\chi)$. This was the approach used by Kontsevich in his seminal talk \cite{IAS}.
\end{remark}

\medbreak\noindent\textbf{Acknowledgments:} The author is very grateful to Lars Hesselholt for useful discussions concerning topological periodic cyclic homology. The author is also thankful to the Mittag-Leffler Institute for its hospitality.

\end{document}

\end{proof}